\documentclass[runningheads,a4paper]{llncs}

\usepackage{hyperref}
\usepackage{csquotes}
\usepackage{mathtools}
\usepackage{amssymb}

\usepackage{url}
\newcommand{\keywords}[1]{\par\addvspace\baselineskip
\noindent\keywordname\enspace\ignorespaces#1}

\newcommand\thlink[1]{\href{http://us.metamath.org/mpegif/#1.html}{\textsf{#1}}}

\newcommand{\down}{{\downarrow}}
\DeclareMathOperator{\card}{card}
\DeclareMathOperator{\dom}{dom}
\DeclareMathOperator{\ran}{ran}
\DeclareMathOperator{\seq}{Seq}
\DeclareMathOperator{\supp}{supp}
\newcommand{\gch}{{\rm GCH}}

\begin{document}

\mainmatter  

\title{GCH implies AC, a Metamath Formalization}


\author{Mario Carneiro}
%

\institute{The Ohio State University, Columbus OH, USA}

\maketitle

\begin{abstract}
We present the formalization of Specker's ``local'' version of the claim that the Generalized Continuum Hypothesis implies the Axiom of Choice, with particular attention to some extra complications which were glossed over in the original informal proof, specifically for ``canonical'' constructions and Cantor's normal form.
\end{abstract}

\section{Introduction}\label{sec:intro}

The Metamath system, consisting of a formal proof language and computer verification software, was developed for the purpose of formalizing mathematics in a minimalistic foundational theory \cite{metamath}. Although Metamath supports arbitrary axiom systems, the main result of this paper was performed within the \texttt{set.mm} database, which formalizes much of the traditional mathematics curriculum into a ZFC-based axiomatization \cite{setmm}. All the theorems in this paper have been formalized and verified for correctness by the Metamath program, and the presence of alternative independently-written verifiers ensure added confidence in the correctness of the proof.

The Generalized Continuum Hypothesis (GCH) is the statement that there are no infinite cardinals $\frak m,\frak n$ in the ordering relationship $\frak m<\frak n<2^\frak m$, and the Axiom of Choice (AC), in one formulation, states that every cardinal is well-orderable. In Metamath, in order to sidestep the complications of defining a cardinal as a set in the absence of choice, we define a cardinal simply as any set, and live with the fact that equality of cardinals is no longer the set-theoretic $x=y$ relation but rather the equinumerosity relation, denoted $x\approx y$.

As both the GCH and AC are of the form $\forall x P(x)$ for an appropriate property $P(x)$, it is meaningful to consider ``local'' versions of each statement. The statement commonly denoted as ${\sf CH}(\frak m)$, means that either $\frak m$ is finite or $\forall\frak n,\break\lnot(\frak m<\frak n<2^\frak m)$. In Metamath we call sets $x$ such that ${\sf CH}(|x|)$ GCH-sets, and define the class of all of them as
$$\gch={\rm Fin}\cup\{x\mid\forall y\,\lnot(x\prec y\land y\prec{\cal P}x)\}$$
where $\rm Fin$ is the class of finite sets (\thlink{df-gch}\footnote{The sans-serif labels mentioned in this paper refer to definitions or theorem statements in \texttt{set.mm}; they can be viewed at e.g. \url{http://us.metamath.org/mpegif/gchac.html} for \thlink{gchac}.}). Thus the Metamath notation $x\in\gch$ corresponds to ${\sf CH}(|x|)$ in the usual notation, and the Generalized Continuum Hypothesis itself is expressed in Metamath notation as $\gch=V$.

The axiom of choice is also expressible in this format. The function $\card(x)$ is defined as the intersection of all ordinals which are equinumerous to $x$, when this intersection exists. Thus $x\in\dom\card$ iff there is an ordinal equinumerous to $x$, which is equivalent to the statement that $x$ is well-orderable; this is used as the standard idiom to express well-orderability, and $\dom\card=V$ is an AC equivalent (\thlink{dfac10}).

\section{GCH implies AC}\label{sec:gchac}

In order to prove $\gch=V\to{\rm AC}$, it is sufficient to prove a statement of the form
\begin{equation}\label{eq:lgch}
\omega\preceq x\land x_1\in\gch\land\dots\land x_k\in\gch\to x\in\dom\card,
\end{equation}
where each $x_i$ is some expression of $x$ which is provably a set assuming $x$ is, because then for any set $y$, setting $x=\omega\cup y$ in (\ref{eq:lgch}) we can prove the assumptions using $\gch=V$ and $\omega\subseteq\omega\cup y$, and then $y$ is well-orderable because it is a subset of the well-orderable set $\omega\cup y$. Thus we call a theorem of the form (\ref{eq:lgch}) a ``local'' form of $\gch\to{\rm AC}$.

The original proof by Sierpi\'{n}ski that GCH implies AC \cite{sierpinski} in fact shows
$$\omega\preceq x\land x\in\gch\land {\cal P}x\in\gch\land {\cal PP}x\in\gch\to x\in\dom\card,$$
and this result was later refined by Specker \cite{specker} to
$$\omega\preceq x\land x\in\gch\land {\cal P}x\in\gch\to \aleph(x)\approx{\cal P}x,$$
where $\aleph(x)$ is the Hartogs number of $x$, the least ordinal which does not inject into $x$. This implies that ${\cal P}x$ and a fortiori $x$ are well-orderable (since $x\preceq{\cal P}x$), and so this is also a local form of $\gch\to{\rm AC}$.

\begin{theorem} [Specker, \textmd{\thlink{gchhar}}]\label{th:gchhar} If $\frak m$ is an infinite cardinal such that ${\sf CH}(\frak m)$ and ${\sf CH}(2^\frak m)$, then $\aleph(\frak m)=2^\frak m$. Or in Metamath notation:
$$\omega\preceq x\land x\in\gch\land {\cal P}x\in\gch\to \aleph(x)\approx{\cal P}x.$$
\end{theorem}

This is the main result of the paper, and the proof follows Kanamori \& Pincus \cite{kanamori} closely (indeed, almost all of \cite{kanamori} was formalized as a result of this project). As the complete argument is presented in formal detail in \thlink{gchhar} and in informal detail in \cite{kanamori}, we will not rehash the details here, but instead point out areas where the informal and formal proofs diverge, indicating places where a full proof is not as simple as it might seem at first glance. The reader is encouraged to consult \cite{kanamori} for an overview of the proof and \thlink{gchhar} for the complete formalization (which is not difficult to read after a little practice).

\section{Canonical Constructions}\label{sec:bpos}

The main divergence from the text proof concerns a certain non-injectibility result, Proposition 1.7 of \cite{kanamori}:

\begin{theorem} [Halbeiben--Shelah, \textmd{\thlink{pwfseq}}]\label{th:gchhar} If $\aleph_0\le|X|$, then $|{\cal P}(X)|\nleq|\seq(X)|$, where $\seq(X)=\bigcup_{n\in\omega}X^n$ is the set of finite sequences on $X$. Or in Metamath notation:
$$\omega\preceq X\to\lnot\,{\cal P}X\preceq\bigcup_{n\in\omega}X^n.$$
\end{theorem}
\begin{proof}
Fix injections $J:\omega\to X$ and $G:{\cal P}X\to\seq(X)$, and suppose we are given an $H$ such that for every infinite well-ordered subset $\langle Y,\sqsubset\rangle,Y\subseteq X$, $H_{Y,\sqsubset}$ is an injection from $\seq(Y)$ to $Y$. Now consider some infinite $\langle Y,\sqsubset\rangle,Y\subseteq X$, and define
$$D_{Y,\sqsubset}=\{x\in Y\mid H_{Y,\sqsubset}^{-1}(x)\in\ran G\land x\notin G^{-1}(H_{Y,\sqsubset}^{-1}(x))\}.$$
Then if $G(D_{Y,\sqsubset})\in\seq(Y)$ one gets the contradiction $$H_{Y,\sqsubset}(G(D_{Y,\sqsubset}))\in D_{Y,\sqsubset}\leftrightarrow H_{Y,\sqsubset}(G(D_{Y,\sqsubset}))\notin D_{Y,\sqsubset},$$ so $D_{Y,\sqsubset}\in\seq(X)\setminus\seq(Y)$, and the minimal element of the sequence not in $Y$ is an element of $X\setminus Y$. Thus we can define $F(Y,\sqsubset)$ to be this element when $Y$ is infinite, and $F(Y,\sqsubset)=J(|Y|)$ when $Y$ is finite, and we will have defined a function from well-orders of subsets of $X$ to elements of $X$ such that \break $F(Y,\sqsubset)\in X\setminus Y$ when $Y$ is infinite. This is the necessary setup for application of Theorem 1.1 of \cite{kanamori} (formalized as \thlink{fpwwe2}), which gives a well-ordered subset $\langle Z,\sqsubset\rangle$ of $X$ satisfying $F(Z,\sqsubset)\in Z$ and $F(x\down,\sqsubset)=x$ for all $x\in Z$ (where $x\down=\{y\in Z\mid y\sqsubset x\}$). If $Z$ is infinite, then this contradicts the definition of $F$, but if $Z$ is finite, then $T=J(|Z|)\down$ is a proper subset of $Z$ such that $J(|T|)=J(|Z|)$, a contradiction.
\qed\end{proof}

There is one hole in this proof, namely the construction of an $H$ such that for every infinite well-ordered subset $\langle Y,\sqsubset\rangle$, $H_{Y,\sqsubset}$ is an injection from $\seq(Y)$ to $Y$. The original proof in \cite{kanamori} has this to say about such a function: \blockquote{For infinite, well-orderable $Y$, we have $|Y|=|\seq(Y)|$; in fact, to every infinite well-ordering of a set $Y$ we can canonically associate a bijection between $Y$ and $\seq(Y)$.}

Given a pairing function on $Y$, by which we mean a bijection $J:Y\times Y\to Y$, and an injection $g:\omega\to Y$, one can construct injections $f_n:Y^n\to Y$ by recursion as $f_0(\emptyset)=g(0)$ and $f_{n+1}(x)=J(f_n(x\restriction n),x(n))$ and define $g_n(x)=\langle n,f_n(x)\rangle$; then since the domain and range of each $g_n$ is disjoint the union of all of them is an injection from $\seq(Y)\to\omega\times Y$, and composing with $J\circ\langle g,I\rangle$ (where $I$ is the identity function) gives an injection $\seq(Y)\to Y$. Thus we are reduced to the question of finding a ``canonical'' pairing function $Y\times Y\to Y$.

This problem can be reduced still further to eliminate the auxiliary well-order $\sqsubset$ of $Y$; since there is a unique isomorphism from $\langle Y,\sqsubset\rangle$ to an ordinal $\alpha<\aleph({\cal P}X)$, we need only find a function that enumerates pairing functions for all ordinals less than $\aleph({\cal P}X)$.

\subsection{Canonical pairing functions}\label{sec:cpf}

The ``classical'' proof of $\alpha\times\alpha\approx\alpha$ (\thlink{infxpen}) using G\"{o}del's pairing function suffers from an inherent nonconstructibility in its approach, because it only produces a legitimate pairing function $\kappa\times\kappa\to\kappa$ when $\kappa$ is an infinite cardinal (or more generally when $\kappa$ is multiplicatively indecomposable), and on other ordinals $\alpha$ one picks(!) some bijection $\alpha\to|\alpha|$ to establish $\alpha\times\alpha\approx\alpha$ generally.

To avoid this, we make use of Cantor normal form, which in our version asserts that the function $f\mapsto\sum_{\gamma\in\supp(f)}\alpha^{\gamma}f(\gamma)$ (where the sum is taken from largest to smallest) is a bijection from the set of finitely supported functions $\beta\to\alpha$ to the ordinal exponential $\alpha^\beta$ (\thlink{cantnff1o}).

Reversing the sum does not preserve the ordinal value, but does preserve its cardinal because $\alpha+\beta\approx\alpha\sqcup\beta$ (where $\sqcup$ is cardinal sum or disjoint union), and similarly for $\alpha\beta\approx\alpha\times\beta\approx\beta\alpha$. Then for any $\beta<\alpha\le\omega^\alpha$, we can write
$$\beta=\sum_{i=1}^n\omega^{\beta_i}k_i\approx\sum_{i=n}^1\omega^{\beta_i}k_i= \omega^{\beta_1}k_1\approx k_1\omega^{\beta_1}=\omega^{\beta_1},$$
where all the equinumerosity relations are witnessed by explicit bijections \break(\thlink{cnfcom3}). This yields a proof that there is a function which enumerates bijections $\beta\to\omega^\gamma$ for some $\gamma(\beta)$ and all $\omega\le\beta<\alpha$, for any upper bound $\alpha$.

We can use this to produce a pairing function using the calculation
$$\beta\times\beta\approx\omega^\gamma\times\omega^\gamma\approx\omega^{\gamma2}
\approx\omega^{2\gamma}=(\omega^2)^\gamma\approx\omega^\gamma\approx\beta,$$
after fixing some bijection $\omega^2\approx\omega\times\omega\approx\omega$, where again the $\approx$ notation is being used as shorthand for an explicit bijection (\thlink{infxpenc}).

\subsubsection*{Acknowledgments.} The author wishes to thank G\'{e}rard Lang for the initial idea for this work and Asaf Karagila for pointing the author in the direction of Cantor normal form as a resolution of the issues in Section~\ref{sec:cpf}.

\end{document}